\documentclass[a4paper]{amsart}

\usepackage{amsthm}
\usepackage{amsmath}
\usepackage{amssymb}
\usepackage{latexsym}
\usepackage{enumerate}
\usepackage{graphicx}
\usepackage[utf8]{inputenc}

\usepackage{epsfig} 
\usepackage{pgf}
\usepackage{tikz}
\usepackage{float}
\usepackage{dsfont}
\usepackage{times}

\newtheorem{theoremA}{Theorem}

\renewcommand{\thetheoremName}

\newtheorem{proposition[[]]}[theoremName]{Proposition G}

\newtheorem{theorem}{Theorem}[section]
\newtheorem{lemma}[theorem]{Lemma}

\newtheorem{proposition}[theorem]{Proposition}
\newtheorem{corollary}[theorem]{Corollary}

\theoremstyle{definition}
\newtheorem{definition}[theorem]{Definition}
\newtheorem{example}[theorem]{Example}

\newtheorem{remark}{Remark}

\numberwithin{equation}{section}

\newcommand{\Hess}{\operatorname{Hess}}
\newcommand{\dist}{\operatorname{dist}}

\newcommand{\kan}{\mathbb{M}^{n}(\kappa)}

\newcommand{\kam}{\mathbb{M}^{m}(\kappa)}
\newcommand{\erre}{\mathbb{R}}

\newcommand{\Han}{\mathbb{H}^n(\kappa)}
\newcommand{\Ha}{\mathbb{H}}
\newcommand{\Ham}{\mathbb{H}^m(\kappa)}

\newcommand{\vol}{\textrm{vol}}

\newcommand{\E}{\mathcal{E}}

\begin{document}

  \title[Asymptotically extrinsic tamed]{Asymptotically extrinsic tamed submanifolds}

\author{G. Pacelli Bessa}      
\address{Departamento de Matem\'atica,  Universidade do Cear\'a-UFC, Fortaleza, CE, Brazil                        
}

\email{bessa@mat.ufc.br}

\author[V. Gimeno]{Vicent Gimeno*}
\address{Departament de Matem\`{a}tiques- IMAC,
Universitat Jaume I, Castell\'{o}, Spain.}
\email{gimenov@uji.es}
\author[V. Palmer]{Vicente Palmer**}
\address{Departament de Matem\`{a}tiques- INIT,
Universitat Jaume I, Castellon, Spain.}
\email{palmer@mat.uji.es}
\thanks{* Work partially supported by the Research Program of University Jaume I Project P1-1B2012-18, and DGI -MINECO grant (FEDER) MTM2013-48371-C2-2-P}
\thanks{**Work partially supported by the Research Program of University Jaume I Project P1-1B2012-18,  DGI -MINECO grant (FEDER) MTM2013-48371-C2-2-P, and Generalitat Valenciana Grant PrometeoII/2014/064 }
\keywords{Volume growth, End, Extrinsic distance, Second fundamental form, Gap theorem, Tamed submanifold}
\subjclass[2010]{Primary 53C20, 53C40; Secondary 53C42}



\begin{abstract} 
We study, from the extrinsic point of view, the structure at infinity of open submanifolds,  $\varphi\colon M^m \hookrightarrow \kan$ isometrically immersed in the real space forms of constant sectional curvature $\kappa \leq 0$.
We shall use the decay of the second fundamental form of the  the so-called {\em tamed} immersions to obtain a description at infinity of the submanifold in the line of the structural results in \cite{GPZ} and \cite{Petrunin2001} and an estimation from below of the number of its ends in terms of the volume growth of a special class of extrinsic domains, the {\em extrinsic balls}.
\end{abstract}

\maketitle

\section{Introduction}

The geometry and the topology in the large of non-compact Riemannian manifolds is controlled by their curvature behavior at infinity, so that one can expect, if the manifold becomes nearly flat at infinity, i.e., out of the compact sets,  that it shares some esential features with the Euclidean space $\erre^n$. This fundamental idea, together with the analysis of asymptotically non-negative curved spaces, was developed in the seminal works \cite{GW1}, \cite{GW2}, \cite{Dre} and \cite{Abresch1985}.

In particular, it has been proved in \cite{GW1}, (resp. in \cite{GW2}), that a complete non-compact  Riemannian manifold $M^n$ with zero sectional curvature outside a compact set, (resp. with non-negative curvature outside a compact set), contains another compact $K \subseteq M$ such that $M \setminus K$ is a finite union of \lq\lq conical ends", each of the form $N \times \erre^+$, being $N$ a connected and compact ($n-1$)-dimensional manifold.

From this point of view, it seems natural to think that it is possible to extract some similar description at infinity of a Riemannian manifold  by replacing the flatness of the manifold outside a compact set by a weaker hypothesis. For instance, we can assume, as in  \cite{GPZ}, that the Riemannian manifold $M$ has {\em faster-than-quadratic-curvature-decay}, namely, that there exists some $\epsilon >0$ and some $c>0$ such that 
$$\vert K_x\vert < c \cdot\rho_M(x)^{-(2+\epsilon)},\,\,\forall x\,\, \text{with} \,\, \rho_M(x) >1$$
 where $K_x$ is the supremum of sectional curvatures of the tangent $2$-planes of $T_xM$ and $\rho_M(x)=\text{dist}_{M}(x_0,x)$ is the distance to a fixed base point $x_0 \in M$. 
 
 Then, it was proved in   \cite[Thm.1]{GPZ} that if
$M$ is a complete, connected and non-compact Riemannian manifold with {\em faster-than-quadratic-curvature-decay}, the manifold  contains a connected open subset $D \subseteq M$ with compact closure and smooth boundary such that the complement $M \setminus D$ is a finite union of \lq\lq conical ends"  $M_i\equiv N_i\times \erre^+$ described as before.  Moreover, if the tangent bundle of each $N_i$ is non-trivial and its fundamental group is finite, the volume growth of the conical end $M_i$ has Euclidean order.

A slightly more general concept   is the notion of {\em asymptotically flateness}. We say that a complete non-compact Riemannian manifold $(M,g)$ is said to be \emph{asymptotically flat} if
\begin{equation*}
A(M)=\limsup_{\rho_M(x)\to\infty}\left\{\vert K_x\vert\cdot\rho^2_M(x)\right\} =0,
\end{equation*}
\noindent 
being $\vert K_x\vert$  and $\rho_{M} (x) = \text{ dist}_{M}(x_0, x)$ as before. One easily checks that $A(M)$ does not depend on the choice of the base point $x_0$ and $A(M)$ is invariant under rescalings of the metric.  

Assuming that the manifold $M$ has  cone structure at infinity, namely, that the pointed Gromov-Hausdorff  limit of a decreasing-to-zero sequence of re-escaled metrics on $(M, g)$ is a metric cone $C$ with vertex $o$, and is assymptoticaly flat, A. Petrunin and W. Tuschmann proved in \cite{Petrunin2001}  an structural result in the line of   \cite[Thm.1]{GPZ}, namely, that there exists an open ball $B_R(p)\subset M$ such that $M\setminus B_R(p)$ is a disjoint union $\cup_iN_i$ of a finite number of ends, \emph{i.e.}, $N_i$ is a connected topological manifold with closed boundary $\partial N_i$ which is homeomorphic to $\partial N_i\times [0,\infty)$. Moreover,  if the end $N_i$ is simply connected, then $N_i$ is homeomorphic to $\mathbb{S}^{m-1}\times [0,\infty)$.

 Note that non-compact manifolds with {\em faster-than-quadratic-curvature-decay} or with  with non-negative curvature has cone structure at infinity, see \cite{Kasue} and \cite{Petrunin2001}.

 We are going, in this paper, to study the structure at infinity of complete non-compact Riemannian manifold   isometrically immersed $\varphi\colon M^m \hookrightarrow \kan$, in the real space forms of constant sectional curvature $\kappa \leq 0$, from an extrinsic point of view. We shall use hence an extrinsic approach, preserving an {\em extrinsic curvature decay condition} satisfied by the so-called {\em tamed} immersions (see Definition \ref{tamed}), given in terms of two extrinsic invariants $a(M)$ and $b(M)$. These invariants describes the decay of the second fundamental form $\alpha$ of the submanifold $M$. We ignore in this extrinsic context the existence of the cone structure at infinity, to obtain a description at infinity of the submanifold in the line of the structural results in \cite{GPZ} and \cite{Petrunin2001} and estimating from below the number of its ends in terms of the volume growth of an special class of extrinsic domains, the {\em extrinsic balls} $D_t(o)= \varphi^{-1}(B^{\kan}_{t}(o))$, where $B^{\kan}_{t}(o)$ denotes the open geodesic ball
of radius $t$ centered at the pole  $o\in\kan$, (see Definition \ref{ExtBall} in Subsection \ref{extdist}).
 
 \begin{theorem}\label{Euclidean1}
Let $\varphi\colon M^m
\hookrightarrow \erre^n$ be an isometric immersion of a complete Riemannian
$m$-manifold $M$ into a $n$-dimensional Euclidean space $\erre^n$. Let $$a(M)= \limsup_{\rho_M(x)\to\infty} \rho_M(x)\Vert \alpha (x)\Vert.$$
\begin{enumerate}
\item If  $a(M) < 1$ then the immersion $\varphi \colon M^m
\hookrightarrow \erre^n$ is proper and $M$ has finite topology. In particular $M$ has finitely many ends, each one of finite topological type. Moreover, there exist an open extrinsic ball $D_R(o)\subset M$ so that $M\setminus D_R(o)$ is a disjoint union $\cup_iV_i$  of ends, and each end $V_i$ is homeomorphic to $\partial V_i \times [0, \infty)$.
\item[]
\item If $m\geq 3$ and    $ a(M) <\frac{1}{2}$   then the (finite) number of ends $\E(M)$ is bounded from below by the volume growth of the extrinsic spheres,
\begin{equation*}
\liminf_{t\to\infty}\frac{\vol(\partial D_t)}{m\omega_mt^{m-1}}\leq \frac{\E(M)}{\left(1-4a(M)^2\right)^\frac{m-1}{2}}
\end{equation*}
and by the volume growth of the extrinsic balls,
\begin{equation*}
\liminf_{t\to\infty}\frac{\vol(D_t)}{\omega_mt^{m}}\leq \frac{1}{\left(1-a(M)^2\right)^{\frac{1}{2}}}\frac{\E(M)}{\left(1-4a(M)^2\right)^\frac{m-1}{2}}\cdot
\end{equation*}

\item If $m\geq 3$, $m$ is odd and 
$$
a(M) <\left[\frac{23-\sqrt{337}}{32}\right]^\frac{1}{2}\approx 0.38,
$$
\noindent then every end $V_i$ is  homeomorphic to $\mathbb{S}^{m-1}\times [0,\infty)$. The  homeomorphism can be strengthened to diffeomorphism if $m\geq 5$.

\item If $m\geq 3$, $m$ is even, 
$$
a(M) <\left[\frac{23-\sqrt{337}}{32}\right]^\frac{1}{2}\approx 0.38,
$$
\noindent and   $ V_i$  is simply connected, then $V_i$ is  homeomorphic to $\mathbb{S}^{m-1}\times [0,\infty)$. The  homeomorphism can be strengthened to diffeomorphism if $m\geq 6$.

\end{enumerate}
\end{theorem}

The hyperbolic version of  Theorem \ref{Euclidean1} is the following theorem.

\begin{theorem}\label{Hyperbolic1}
Let $\varphi\colon \!M^m
\hookrightarrow \Han$ be an isometric immersion of a complete Riemannian
$m$-manifold $M$ into the $n$-dimensional Hyperbolic space $\Han$ with constant sectional curvature  $\kappa <0$. Set $$a(M)=\limsup_{\rho_M(x)\to\infty}\frac{1}{\sqrt{-\kappa}}\tanh(\sqrt{-\kappa}\cdot\rho_{_M}(x))\Vert \alpha (x)\Vert$$ and $$b(M)=\limsup_{\rho_M(x)\to\infty}\frac{1}{\sqrt{-\kappa}}\cosh(\sqrt{-\kappa}\cdot\rho_{_M}(x))\sinh(\sqrt{-\kappa}\cdot\rho_{_M}(x))\Vert \alpha (x)\Vert.$$   Then
\begin{enumerate}
\item If $a(M)<1$ then the immersion is proper and $M$ has finite topology. In particular $M$ has finite ends each one of finite topological type. Moreover, there exist an open extrinsic ball $D_R(o)\subset M$ such that $M\setminus D_R(o)$ is a disjount union $\cup_iV_i$ of a finite number of ends, and $V_i$ is homeomorphic to $\partial V_i \times [0, \infty)$.
\item[]

\item 
 If $b(M)<\infty$ and $m \geq 3$, then the (finite) number of ends $\E(M)$ are bounded from below by the volume growth of the extrinsic spheres
\begin{equation*}
\begin{aligned}
\liminf_{t\to \infty}\frac{\vol (\partial D_t)}{\vol(S_t^{\kappa,m-1})}\leq \E(M)
\end{aligned}
\end{equation*}
and by the  volume growth of the extrinsic balls
\begin{equation*}
\begin{aligned}
\liminf_{t\to \infty}\frac{\vol (D_t)}{\vol(B_t^{\kappa,m})}\leq \E(M).
\end{aligned}
\end{equation*}
where $B^{\kappa, m}_t$ and  $S^{\kappa, m-1}_t$ are the geodesic $t$-ball and the geodesic $t$-sphere of radius $t$ in $\Ham$ respectively. 
Moreover,  the fundamental tone $\lambda^*(M)$ is bounded from above by the fundamental tone $\lambda^*(\Ham)$ of the hyperbolic space $\Ham$, \emph{i.e.}, 
\begin{equation*}\label{tone-above}
\lambda^*(M)\leq \frac{-\left(m-1\right)^2\kappa}{4}
\end{equation*}
\end{enumerate}
\end{theorem}

We observe here that structural statement (1) in Theorems \ref{Euclidean1} and \ref{Hyperbolic1} comes directly from the following Theorem \ref{tamed-theorem},  first stated in \cite{Pac} for the case $\kappa=0$, in  \cite{Pac2} for the case $\kappa <0$ and then in \cite{GPGap} it was given an extension of it to complete ambient manifolds with a pole and bounded radial curvatures. Theorem  \ref{tamed-theorem} constitutes an extrinsic version of the structural assertion in  \cite[Thm.1]{GPZ} and  in \cite[Thm. A]{Petrunin2001} for the special class of submanifolds in $\kan$ called tamed submanifolds.

\begin{theoremA}[{\cite{Pac2, Pac, GPGap}}]\label{tamed-theorem}
Let $\varphi\colon \!M^m\! \hookrightarrow\! \kan$ be an isometric immersion  of a complete Riemannian
$m$-manifold $M$ into an  $n$-dimensional space form  $\kan$ with constant sectional curvature $\kappa\leq 0$. Let us suppose that
$$a(M)=\limsup_{\rho_M(x)\to\infty}\frac{1}{\sqrt{-\kappa}}\tanh(\sqrt{-\kappa}\rho_M(x))\Vert \alpha (x)\Vert <1$$ 

\noindent Then:
\begin{enumerate}
\item $\varphi$ is proper.
\item $M$ has finite topology.
\item There exist $R_0\in M$ such that the extrinsic distance function has no critical points in $M\setminus D_{R_0}$, where  $D_R(x_0)$ denotes the extrinsic ball of radius $R$ centered at $x_0\in M$.

\item In particular, $M\setminus D_{R_0}$ is a disjount union  $\cup_k V_k$ of finite number of ends. $M$ has so many ends $\E(M)$ as components $\partial D_{R_0}$ has , and each end $V_k$ is diffeomorphic to $\partial D_{R_0}^k\times [0,\infty)$, where $\partial D_{R_0}^k$ denotes the component of $\partial D_{R_0}$ which belongs to $V_k$.  
\end{enumerate}
\end{theoremA}

In the main theorem of \cite{Petrunin2001}, above mentioned, it was also proved that if $M^m$, $m \geq 3$, has  cone structure at infinity, is asymptotically flat and  is simply connected with non-negative sectional curvature then $M$ is isometric to $\erre^m$.

This gap result for manifolds with non-negative sectional curvatures, gives a partial answer (assuming the additional hypothesis that the manifold has cone structure at infinity) to the problem posed by M. Gromov in \cite{BallGro}:
\medskip

{\em If $M$ is simply connected of dimension $n \geq 3$ and asymptotically flat with non-negative curvature, show that $M$ is isometric to $\erre^n$}.

\medskip

 Greene and Wu \cite{GW2}, adressed  this question when the manifold $M$ has a pole  showing that in this case and when $M$ has faster-than-quadratic-curvature-decay, the manifold is isometric to $\erre^n$. From an extrinsic  point of view,  Kasue and Sugahara \cite{KS}, established the following gap result:

\begin{theoremA}\label{kasue}(\cite{KS})

\medskip

(I) Let $\varphi \colon M^m\hookrightarrow \erre^n$ be a connected, non-compact Riemannian submanifold  properly immersed into $\erre^n$. Suppose that $M$ has one end and the second fundamental form of the immersion satisfies
$$ \sup \rho_M^{\alpha}(x)\Vert \alpha (x)\Vert < \infty$$
for a constant $\alpha >2$.

Then $M$ is totally geodesic if $2m>n$ and the sectional curvature is non-positive everywhere on $M$, or if $m=n-1$ and the scalar curvature is non-positive everywhere on $M$.
\medskip

(II) Let $\varphi \colon M^m \hookrightarrow \Han$ be a connected, non-compact Riemannian submanifold  properly immersed into $\Han$. Suppose that $M$ has one end and
$$ e^{2 \rho_M(x)}\Vert \alpha (x)\Vert \longrightarrow 0$$
as $x \in M$ goes to infinity. Then $M$ is totally geodesic if $2m>n$ and the sectional curvature is everywhere less than or equal to $\kappa$ or if $m=n-1$ and the scalar curvature is everywhere less than or equal to $m(m-1)\kappa$.
\end{theoremA}

We can state the following gap type result that improves Kasue-Sugahara's results in \cite{KS} and extends Greene-Wu's gap to submanifolds of Hyperbolic space. This theorem is proved as a corollary of the proofs of Theorems \ref{Euclidean1} and \ref{Hyperbolic1}.

\begin{theorem}\label{gap1}
Let $\varphi\colon\! M^m\! \hookrightarrow \kan$ be an isometric immersion of a complete Riemannian
$m$-manifold $M$, $m\geq 3$ into a $n$-dimensional space form  $\kan$ with constant sectional curvature $\kappa\leq 0$. Suppose that  $M$ is simply connected with sectional curvatures $K_M\leq \kappa$. Then

(a) If  $\kappa=0$ and $a(M)=0$, $M$ is isometric to $\erre^m$.

(b) If $\kappa< 0$ and $b(M)<\infty$, $M$ is isometric to $\Ham$.
\end{theorem}

Concerning the assertions (2) and (3) in Theorems \ref{Euclidean1} and \ref{Hyperbolic1}, V. Gimeno and V. Palmer in  \cite{GPGap}, proved  that there is a deep relation between the volume growth of the extrinsic spheres and the number of ends of extrinsic asymptotically flat   submanifolds of rotationally symmetric spaces. In the particular setting of minimal immersions of the Euclidean space they showed that 

\begin{theoremA}[See \cite{GPGap}]\label{GimPal2}
Let $\varphi\colon M^m
\hookrightarrow \erre^n$ be  an isometric and minimal immersion of a complete Riemannian
$m$-manifold $M$ into the $n$-dimensional Euclidean space $\erre^n$. If  $a(M)=0$ and $m\geq 3$,  the (finite) number of ends $\E(M)$ is bounded from below by
\begin{equation*}
\lim_{t\to\infty}\frac{\vol(D_t)}{\omega_mt^{m}}\leq \E(M),
\end{equation*}
If  $M$ has only one end then $M$ is isometric to $\erre^m$. 
\end{theoremA}

Theorem \ref{GimPal2} shows a relation between the volume growth of the extrinsic balls and the number of ends, and moreover one deduce a gap type theorem first stated by A. Kasue and K. Sugahara in \cite{KS}. 

However, this gap result does not hold for minimal submanifolds of the Hyperbolic space, as we can see in the following example, given in \cite{Mari}.

\begin{example}
 In \cite{Mari}, the authors consider a minimal graph $M^{n} \subseteq \Ha^{n+1}$ over a bounded and regular domain $\Omega \subseteq \partial_{\infty}\Ha^{n+1}$, proving that $M$ has finite total (extrinsic) curvature i.e. $ \int_M \Vert \alpha_M \Vert^m d\sigma < \infty$. Then, applying Lemma 3.1 in the proof of \cite[Thm. A]{O}, we have that $\Vert \alpha_M(x)\Vert$ goes to zero when $\rho_M(x)$ goes to infinity and hence, $a(M)=0$.

We can conclude from this fact that, in the case of minimal submanifolds of Hyperbolic space, to be extrinsically asymptotically flat (i.e., to have the curvature decay $a(M)=0$) it is not enough to characterize the hyperbolic subspaces, justifying the introduction of the invariant $b(M)$ and the extrinsic curvature decay criterion $b(M)< \infty$. 
\end{example}

Taking into account the relation between the fundamental tone $\lambda^*(M)$ and the Cheeger isporimetric constant $\mathcal{I}(M)$ (see \cite[theorem 3, chap. IV]{Chavel} for instance), inequality (\ref{tone-above}) implies the following result for minimal immersions in the Hyperbolic space 

\begin{corollary}
Let $\varphi\colon\! M^m\!
\hookrightarrow\! \Han$ be an isometric immersion of a complete Riemannian
$m$-manifold $M$ into the $n$-dimensional Hyperbolic space $\Han$ with constant sectional curvature  $\kappa <0$ and let us suppose that $m\geq 3$ and $b(M)<\infty$. Then

 \begin{equation*}
\mathcal{I}(M)\leq (m-1)\sqrt{-\kappa}.
\end{equation*}

\end{corollary}

In the particular setting of minimal immersions of Hyperbolic space, using the lower bounds for the Cheeger constant and the fundamental tone for minimal submanifolds in $\Han$ given in \cite{GPCheeger}, we can state an improved version of the theorems  \cite[Thm. 1.1]{GPGap} and   \cite[Thm. B]{GimFT}.

\begin{corollary}
Let $\varphi\colon\! M^m\!
\hookrightarrow\! \Han$ be a minimal immersion of a complete Riemannian
$m$-manifold $M$ into the $n$-dimensional Hyperbolic space $\Han$ with constant sectional curvature  $\kappa <0$ and let us suppose that $m\geq 3$ and $b(M)<\infty$. Then
\begin{enumerate}
\item $M$ has finite topological type and  the (finite) number of ends $\E(M)$ is bounded from below by
\begin{equation*}
\begin{aligned}
\sup_{t\in\erre_+}\frac{\vol (D_t)}{\vol(B_t^{\kappa,m})}\leq \E(M).
\end{aligned}
\end{equation*}
\item The fundamental tone $\lambda^*(M)$ satisfies
\begin{equation*}
 \lambda^*(M)=\frac{-\left(m-1\right)^2\kappa}{4}
\end{equation*}
\item The Cheeger constant satisfies
\begin{equation*}
\mathcal{I}(M)= (m-1)\sqrt{-\kappa}
\end{equation*}
\item If $M$ has only one end, $(\E(M)=1)$, them $M$ is isometric to $\Ham$.
\end{enumerate}
\end{corollary}

\subsection{Outline of the paper} The structure of the paper is as follows:

In the preliminaries, Section \S \ref{prelim}, subsection \S \ref{extdist}, we recall the preliminary concepts and properties of extrinsic distance function. In subsection \S \ref{subsec2.2} it is presented and studied the notion of {\em tamed submanifold} and we finish the preliminaries establishing lower and upper bounds for the sectional curvatures of the boundary of an end in a tamed submanifold, in subsection \S \ref{subsec2.3}.
We shall prove theorem \ref{Euclidean1} in  Section \S \ref{section-proof-euc}, obtaining as a result of that proof Corollaries \ref{dim3-theo} and \ref{betti} which deal about several topological properties of the ends of the submanifold, such as vanishing first Betti number. We prove Theorem \ref{Hyperbolic1}  in \S \ref{section-proof-hyp}, obtaining Corollaries \ref{dim3-theoHyp} and \ref{bettiHyp} in the same way as in Section \S \ref{section-proof-euc}. Finally, in \S \ref{section-proof-gap} the gap type result,  Theorem \ref{gap1}, is proved.

\section{Preliminaries}\label{prelim}

\subsection{Analysis of the extrinsic distance function defined on a submanifold}\label{extdist}
We start  presenting some standard definitions and results that we can find in previous works (see e.g. \cite{GPGap}, \cite{Pa}). We assume throughout the paper that $\varphi\colon M \hookrightarrow \kan$ is an isometric immersion of a complete non-compact Riemannian $m$-manifold $M$ into a $n$-dimensional real space form $\kan$ of constant sectional curvature $\kappa\leq 0$. For every $x \in \kan\setminus \{o\}$ we
define $r(x) = r_o(x) = \dist_{\kan}(o, x)$, and this
distance is realized by the length of a unique
geodesic from $o$ to $x$, which is the {\it
radial geodesic from $o$}. We also denote by $r\vert_M$ or by $r$
the composition $r\circ \varphi: M\to \erre_{+} \cup
\{0\}$. This composition is called the
{\em{extrinsic distance function}} from $o$ in
$M$. The gradients of $r$ in $\kan$ and of $r\vert_M$ in  $M$ are
denoted by $\nabla^{\kan} r$ and $\nabla^M r$,
respectively. Then we have
the following basic relation, by virtue of the identification, given any point $x\in M$, between the tangent vectors $X \in T_xM$ and $\varphi_{*_{x}}(X) \in T_{\varphi(x)}\kan$
\begin{equation}\label{radiality}
\nabla^{\kan} r = \nabla^M r +\nabla^\bot r ,
\end{equation}
where $\nabla^\bot r(\varphi(x)):=(\nabla^M r)^\bot(\varphi(x))$ is perpendicular to $T_{x}M$ for all $x\in M$.

\begin{definition}\label{ExtBall}
Given $\varphi: M^m \longrightarrow \kan$ an isometric immersion of a complete and connected Riemannian $m$-manifold $M$ into a real space form $\kan$ of constant sectional curvature $\kappa\leq 0$,  we define the {\em{extrinsic metric balls}} of radius $t >0$ and center $o \in \kan$ as  the subsets of $P$:
\begin{equation*}
\begin{aligned}
D_t(o)&=\{x\in M : r(\varphi(x))< t\}\\&=\{x\in M : \varphi(x) \in B^{\kan}_{t}(o)\} = \varphi^{-1}(B^{\kan}_{t}(o))
\end{aligned}
\end{equation*}
where $B^{\kan}_{t}(o)$ denotes the open geodesic ball
of radius $t$ centered at $o\in\kan$. 
\end{definition}

\begin{remark}\label{theRemk0}
Despite the set $\varphi^{-1}(o)$ in the above definition can be the empty set, in this paper we always chose an $o\in \kan$ such that $\varphi^{-1}(o)=\{q\}$.
When the immersion $\varphi$ is proper, the extrinsic domains $D_t(o)$
are precompact sets, with smooth boundary $\partial D_t(o)$.  The assumption on the smoothness of
$\partial D_{t}(o)$ makes no restriction. Indeed, 
the distance function $r$ is smooth in $\kan\setminus \{o\}$. Hence
the composition $r\vert_M$ is smooth in $M\setminus \{q\}$ and consequently the
radii $t$ that produce smooth boundaries
$\partial D_{t}(o)$ are dense in $\mathbb{R}_+$ by
Sard's Theorem and the Regular Level Set Theorem.
\end{remark}

\begin{proposition}
Let $\varphi\colon M\to N$ be an isometric immersion of a  Riemannian manifold $M$ into a Riemannian manifold $N$ and let $f\colon N\to \erre$ be a smooth function, then
\begin{equation}\label{hessian-com}
\Hess^M(f\circ \varphi)(u,v)=\Hess^Nf(\varphi_*(u),\varphi_*(v))+\langle\nabla^Nf,\alpha(u,v)\rangle,
\end{equation}where $\alpha$ is the second fundamental form of the immersion.
\end{proposition}

On the other hand, the Hessian of the distance function $r\colon\kan\setminus\{0\}\to \erre$ at a point $p \in \kan$ is given by 
\begin{proposition}\label{hess-dist}
\begin{equation}
\Hess_p^{\kan} r(u,v)=\frac{C_\kappa}{S_\kappa}(r(p))\left(\langle u, v\rangle-\langle \nabla^{\kan} r, u\rangle\langle \nabla^{\kan} r, v\rangle\right).
\end{equation}
where the function $S_{\kappa}$ is given by

\begin{equation}\label{eqSk}
 S_{\kappa} (t)=\left \{
\begin{array}{ccl}
\displaystyle \frac {1}{\sqrt{-\kappa}}\sinh(\sqrt{ -\kappa}\,t),&if& \kappa<0 \\
t, &if& \kappa=0
\end{array} \right.
\end{equation}
and $C_{\kappa}(t)= S_{\kappa}'(t)$.
\end{proposition}
Let us recall that, if $\varphi\colon M^m
\hookrightarrow \kan$ is a tamed isometric immersion of a complete Riemannian
$m$-manifold $M$ into the  $n$-dimensional real space form $\kan$ with constant sectional curvature  $\kappa \leq 0$, then there exist $R_0\in M$ such that the extrinsic distance function has no critical points in $M\setminus D_{R_0}(x_0)$, where  $D_R(x_0)$ denotes the extrinsic ball of radius $R$ centered at $x_0\in M$. We have the following technical result in this context:

\begin{lemma}\label{kasue}{{\cite[Proof of lemma 4]{KS}}} Let $\varphi\colon M^m
\hookrightarrow \kan$ be an isometric immersion of a complete Riemannian
$m$-manifold $M$ into a $n$-dimensional real space form $\kan$ with constant sectional curvature  $\kappa \leq 0$. Let us suppose that there exists $R_0>0$ such that the extrinsic distance function has no critical points in $M\setminus D_{R_0}(x_0)$.
Suppose that there exist a function $G:\erre\to\erre$ such that $\Vert \alpha \Vert (x)\leq G(r(x))$. Then for any $x\in M\setminus D_{R_0}(x_0)$, 
\begin{equation}
 \vert \nabla^\perp r\vert \leq \delta(r(x))+ \frac{1}{S_{\kappa}(r(x))}\int_{R_0}^{r(x)}S_{\kappa}(s)G(s)ds.
\end{equation}
Here $\delta(t)$ is a decreasing function such that $\delta\to 0$ when $t\to\infty$.
\end{lemma}

\subsection{{\em Tamed} submanifolds. Some examples}\label{subsec2.2}
The extrinsic decay conditions in the results stated above, can be described more carefully in the following way: 

\begin{definition}
Let $\varphi\colon\! M^{m}\!
\hookrightarrow \kan$ be an isometric immersion of a complete Riemannian
$m$-manifold $M$ into the $n$-dimensional Euclidean space $\erre^n$ or the Hyperbolic space $\Han$. Fix a point $x_0 \in M$ and let $\rho_{M} (x) = {\rm dist}_{M}(x_0, x)$ be the  distance function on $M$ to $x_0$. Let
$\{C_{i}\}_{i=1}^{\infty}$ be a nested exhaustion sequence of $M$ by
compacts sets with $x_0 \in C_0$. Let $\{a_{i}\}_{i=1}^{\infty}\subset [0,\infty]$ and $\{b_{i}\}_{i=1}^{\infty}\subset [0,\infty]$ two sequences defined as

\begin{equation}
\begin{aligned}
a_i =& \sup \Biggl \{ \displaystyle
\left(\frac{S_{\kappa}}{C_{\kappa}}\right)(\rho_{M} (x))\cdot
\Vert \alpha (x)\Vert, \, x \in M \backslash C_i \Biggr \}\,\,\,\forall i=1,...,\infty\\
b_i = & \sup \Biggl \{ \displaystyle
\left(C_{\kappa} \cdot S_{\kappa}\right)(\rho_{M} (x))\cdot
\Vert \alpha (x)\Vert, \, x \in M \backslash C_i \Biggr \}\,\,\,\forall i=1,...,\infty
\end{aligned}
\end{equation}
where   $\Vert \alpha(x)\Vert $ is the  norm of the second
fundamental form at $\varphi(x)$.

With those two sequences we define
\begin{equation}
\begin{aligned}
a(M):=&\lim_{i\to\infty}a_i\\
b(M):=&\lim_{i\to\infty}b_i
\end{aligned}
\end{equation}
The numbers $a(M)$ and $b(M)$
does not depend on the exhaustion sequence $\{C_{i}\}$ nor on the base point $x_{0}$. 
\end{definition}

With the extrinsic invariants $a(M)$ and $b(M)$ in hand, we define the following {\em extrinsic curvature decays}

\begin{definition}\label{tamed}
An immersion $\varphi\colon \! M^m \hookrightarrow \kan$ of a complete Riemannian
$m$-manifold $M$ into a $n$-dimensional space form  $\kan$ with constant sectional curvature $\kappa\leq 0$ has {\em tamed} second fundamental form, (or simply, it is {\em tamed}) if and only if $a(M)<1$. When $a(M)=0$, then $M$ is {\em extrinsically asymptotically flat}. In the case $\varphi\colon M \hookrightarrow \Han$ we say that $M$ is {\em strongly tamed} when $b(M) < \infty$.
\end{definition}

\begin{remark}
Note that for immersions $\varphi\colon M^m \hookrightarrow \Han$, $b(M) < \infty$ implies $a(M) <1$, (in fact, $a(M)=0$, see Remark \ref{f} below) i.e., to be strongly tamed implies to be tamed.
\end{remark}

 We are going to give some examples and remarks which could help to understand better the notion of asymptotically flateness.
 
 \begin{remark}
For   isometric immersions $\varphi\colon \!M^m\!
\hookrightarrow \erre^n$ of  complete Riemannian
$m$-manifolds $M$ into the $n$-dimensional Euclidean space we have that   $a_i(M)=b_i(M)$ for all $i$. Hence, in the Euclidean case we only consider an invariant, $a(M)$.
\end{remark}
\begin{remark}
  Observe also that, for   isometric immersions $\varphi\colon M^m
\hookrightarrow \erre^n$ of  complete Riemannian
$m$-manifolds $M$ into the $n$-dimensional Euclidean space, $a(M)=0$ implies $A(M)=0$ because using the Gauss formula we have that the second fundamental form of $M$ satisfies the inequalities 

\begin{equation*}
-2\Vert\alpha(x)\Vert^2\leq \vert K_x\vert \leq \Vert \alpha(x)\Vert^2.
\end{equation*}

\noindent However, the opposite implication it is not true in general, as we shall show below.
\end{remark}
\begin{remark}
For isometric immersions $\varphi\colon M^m
\hookrightarrow \Han$ of  complete Riemannian
$m$-manifolds $M$ into the $n$-dimensional Hyperbolic space $\Han$,  $a(M)=0$ does not imply $A(M)=0$. For example, let   $\Ham \subseteq \Han$ be a  totally geodesic immersion. In this case, $a_i(\Ham)=b_i(\Ham), \,\,\forall i$, so $a(\Ham)=b(\Ham)=0$. However, $A(\Ham)=\infty$, because $\vert K_x \vert= -\kappa >0$.
\end{remark}

\begin{remark}\label{f}
Also in the case of an isometric immersion of a complete Riemannian
$m$-manifold $M$ into the $n$-dimensional Hyperbolic space $\varphi\colon M^m
\hookrightarrow \Han$, we have that  $b(M) < \infty$ implies $a(M)=0$. To see it, note that $b(M)= \lim_{i\to\infty}b_i <\infty$ implies in this case that $\frac {1}{\sqrt{-\kappa}}\sinh(\sqrt{ -\kappa}\,\rho_M(x))\cosh(\sqrt{ -\kappa}\,\rho_M(x))\Vert \alpha (x)\Vert$ is finite when $\rho_M(x)$ goes to infinity, so $\Vert \alpha (x)\Vert$ goes very fast to zero when $\rho_M(x)$ goes to infinity and this implies, as \[\lim_{\rho_M(x) \to \infty}  \frac{\left(\displaystyle\frac{S_{\kappa}}{C_{\kappa}}\right)(\rho_{M} (x))}{\left(C_{\kappa} \cdot S_{\kappa}\right)(\rho_{M} (x))}=0,\] that $a(M)=0$.
\end{remark}

\begin{example} We have seen that, when we consider an isometric immersion $\varphi: M
\hookrightarrow \erre^n$ then $a(M)=0$ implies that $A(M)=0$.
However, the opposite implication it is not true in general. If we consider the cylinder $C=\{(x,y,z) \in \erre^3/ x^2+y^2=1\} \subseteq \erre^3$ isometrically immersed by the inclusion map, we know that its sectional curvature is $K^C_p(\sigma)=0$ for all points $p \in C$ and all tangent planes $\sigma \subseteq T_pC$. Hence, $A(C)=0$. On the other hand, the norm of its second fundamental form $\Vert \alpha^C \Vert=\text{constant}$, so $a(C)=\infty$.

\end{example}

\begin{example}
Extrinsic asymptotically flateness $a(M)=0$ implies intrinsic asymptotic flateness $A(M)=0$ for submanifolds of $\erre^n$, and, in any ambient space form $\kan$, if the submanifold is extrinsically asymptotically flat, then it is tamed. Observe too that in the hyperbolic space, submanifolds with $a(M)<1$ or $b(M)<\infty$ are not in general asymptotically flat, (although in this case, we have seen that $b(M) < \infty$ implies $a(M)=0$,\,  i. e., the manifold is extrinsically asymptotically flat). Consider for instance the totally geodesic immersion $\varphi:\Ham \hookrightarrow \Han$, which  has $a(\Ham)=b(\Ham)=0$ but with  $A(\Ham)=\infty$. 
\end{example}

\begin{example}
We are going to present, following the construction given in \cite{DoCDa}, a rotation hypersurface $M^n$ of $\Ha^{n+1}(-1)$, $n \geq 2$, with $b(M) < \infty$. For that, let us consider first the Hyperbolic space $\Ha^{n+1}(-1)$ as a hypersurface of the Lorentzian space $L^{n+2}$, with Lorentzian metric $g_{-1}$.

Let us choose $P^2$ a $2$-dimensional plane in $L^{n+1}$, passing through the origin and such that the restriction $g_{-1}\vert_{P^2}$ is Lorentzian. Let us denote as $O(P^2)$ the set of all orthogonal transformations of $L^{n+2}$ with positive determinant and such that leaves $P^2$ fixed. Then, let us consider now a subspace $P^3 \subseteq L^{n+2}$ such that $P^2 \subseteq P^3$ and $P^3 \cap \Ha^{n+1}(-1) \neq \emptyset$ and finally, let $C$ be a regular curve in $P^3 \cap \Ha^{n+1}(-1) $ that does not meet $P^2$. With all this elements in hand, we define the {\em rotation hypersurface $M^n \subseteq \Ha^{n+1}(-1)$} as the orbit of $C$ under the action of $O(P^2)$.

To give an explicit parametrization of this submanifold we start describing the set $O(P^2)$. We choose, always following \cite{DoCDa}, an orthonormal basis $\{e_1,...,e_{n+1},e_{n+2}\}$ of $L^{n+2}$ such that $P^2$ is the plane generated by $\{e_{n+1}, e_{n+2}\}$, $g_{-1}(e_{n+2},e_{n+2})=-1$ and the matrix of an element of $O(P^2)$ is written as a block diagonal matrix, having square matrices $A_i$,  ($i=1,...,n/2 +1$ if $n$ is even, $i=1,..., (n-1)/2 +1$ if $n$ is odd), as main diagonal blocks. Each of these square matrices $A_i$ corresponds with a rotation of angle $\theta_i$, with $\theta_{n/2 +1} =0$.

Let $P^3$ be the space generated by $\{e_1,e_{n+1}, e_{n+2}\}$. We have that $P^2 \subseteq P^3$ and we parametrize the curve $C$ in $P^3 \cap \Ha^{n+1}$ by $(x_1(s), 0, ..., 0,x_{n+1}(s), x_{n+2}(s))$. Now, given a fixed $s=s_0$, let us consider the point $(x_1(s_0),0, ..., 0, x_{n+1}(s_0), x_{n+2}(s_0)) \in C$, and then the orbit under $O(P^2)$ of this point , denoted as $U(s_0)$. $U(s_0)$ is a sphere obtained as the intersection of an affine plane parallel to $\langle\{e_1,...,e_n\}\rangle$  with $\Ha^{n+1}$ and is the parallel of $M^n$ passing through the point $(x_1(s_0),0, ..., 0, x_{n+1}(s_0), x_{n+2}(s_0))$. Hence, a parametrization of $M$ can be obtained parametrizying this orbit $U(s_0)$ and letting $s_0$ vary.

Let us consider $\varphi(t_1,...,t_{n-1})=(\varphi_1(t_1,...,t_{n-1}),...,\varphi_n(t_1,...,t_{n-1}))$ an orthogonal parametrization of the unit sphere of $\langle \{e_1,...,e_n\} \rangle$. Then, 
\begin{equation}
f(t_1,...,t_{n-1},s):=(x_1(s)\varphi_1,...,x_1(s)\varphi_n, x_{n+1}(s), x_{n+2}(s))
\end{equation}
\noindent is the parametrization of the rotation hypersurface $M^n$ generated by the curve $C$ around $P^2$.
In Proposition 3.2 in \cite{DoCDa} it is proved that the principal curvatures of $M$ along the principal directions given by the coordinate curves $t_i$, ($i=1,...,n-1$) are
\begin{equation}
\lambda_i(s)=-\frac{\sqrt{1+x_1^2-\dot{x}_1^2}}{x_1}=\lambda(s)
\end{equation}
\noindent and  the principal curvature along the coordinate curve $s$
is

\begin{equation}
\mu(s)=\frac{\ddot{x}_1-x_1}{\sqrt{1+x_1^2-\dot{x}_1^2}}
\end{equation}

When $x_1(s)=(a\cosh(2s)-\frac{1}{2})^{\frac{1}{2}}$, with $a \in \erre$, $a >\frac{1}{2}$, it is straightforward to check that $\lambda=-\mu$, so when the dimension of the submanifold $M^n$ is $n=2$, then $M$ is minimal, and, computing the norm of the second fundamental form of $M^2$, denoted as $\alpha_{M^2}$, we have, see \cite{DoCDa}, that 
\begin{equation}
\Vert \alpha_{M^2}(s) \Vert^2=\lambda^2(s)+\mu^2(s)=2\lambda^2=\frac{2(a^2-\frac{1}{4})}{(a \cosh(2s)-\frac{1}{2})^2}
\end{equation}
so 
\[\lambda^2(s)=\frac{(a^2-\frac{1}{4})}{(a \cosh(2s)-\frac{1}{2})^2}\cdot \]

Hence, if we consider now $M^n$ with $n >2$, we obtain

\begin{equation}
\Vert \alpha_{M^n}(s) \Vert^2=n\lambda^2=n\frac{2(a^2-\frac{1}{4})}{(a \cosh(2s)-\frac{1}{2})^2}\cdot
\end{equation}

We can compute $b(M^n)$ for these hypersurfaces, ($n \geq 2$), applying L'Hospital's rule:
\begin{equation}
\begin{aligned}
b(M^n)&=\lim_{i\to\infty}b_i=\lim_{s\to\infty}\cosh s\sinh s\cdot
\Vert \alpha_M(s)\Vert\\=&\sqrt{n(a^2-\frac{1}{4})}\lim_{s\to\infty}\frac{\frac{1}{2}\sinh 2s}{a\cosh 2s-\frac{1}{2}}\\&=\sqrt{n(a^2-\frac{1}{4})}\lim_{s\to\infty}\frac{\cosh 2s}{2a\sinh 2s}= \frac{\sqrt{n(a^2-\frac{1}{4})}}{2a} <\infty.
\end{aligned}
 \end{equation}

\end{example}

\subsection{Sectional curvature of the extrinsic spheres in tamed submanifolds}\label{subsec2.3}
Let $M$ be a non-compact Riemannian manifold and $K\subset M$ a compact subset $K\subset M$. An \emph{end $V$ of $M$ with respect to $K$} is an unbounded connected  component of $M\setminus K$. 

\begin{proposition}\label{curvature-comp}
Let $\varphi\colon \!M\!\to \!\kan$ be an  isometric immersion with  $a(M)<1$.  Then, for any end $V$ with respect to a compact set $K\subset M$  there exist $t_0\in \erre_+$ (independent of the end $V$) such that  the sectional curvatures $K_{\partial V(t)}(\pi)$ of the planes $\pi\subset T_p\partial V(t)$ tangents to the extrinsic spheres $\partial V(t):=\partial D_t\cap V$ are bounded form above and below by
\begin{equation}
K_{\partial V(t)}(\pi)\leq \kappa +\Vert \alpha \Vert^2+\frac{\left(\frac{C_\kappa}{S_\kappa}(t)+\vert \nabla^\perp r\vert\, \Vert\alpha\Vert\right)^2}{\vert \nabla^M r\vert^2},
\end{equation}

\begin{equation}
K_{\partial V(t)}(\pi)\geq \kappa -2\Vert \alpha \Vert^2+\frac{\left(\frac{C_\kappa}{S_\kappa}(t)\right)^2-2\vert \nabla^\perp r\vert\, \Vert\alpha\Vert\frac{C_\kappa}{S_\kappa}(t)}{\vert \nabla^M r\vert^2},
\end{equation}

for all $t>t_0$.
\end{proposition}
\begin{proof}

  Suppose that $e_i,e_j$ are two orthonormal vectors of $T_p\partial V(t)$ at a point $p\in \partial V(t)$. Then the sectional curvature $ K_{\partial V(t)}(\pi)$ of the plane $\pi$ expanded by $e_i,e_j$ is, using Gauss formula, see \cite{GPGap}:
  \begin{equation}
  \begin{aligned}
  K_{\partial V(t)}(\pi)=& \kappa+\langle\alpha(e_i,e_i),\alpha(e_j,e_j)\rangle-\vert\alpha(e_i,e_j)\vert^2\\&+ \langle \alpha^{\partial V-M}(e_i,e_i),\alpha^{\partial V-M}(e_j,e_j)\rangle-\vert \alpha^{\partial V-M}(e_i,e_j)\vert ^2
  \end{aligned}
  \end{equation}

\noindent where $\alpha^{\partial V-M}$ is the second fundamental form of $\partial V(t)$ in $M$. Computing $\alpha^{\partial V-M}(x,y)$ for any two vectors in $T_p\partial V(t)$, using Propositions \ref{hessian-com} and \ref{hess-dist} we obtain,
\begin{equation}
\begin{aligned}
\alpha^{\partial V-M}(x,y)=& -\langle x,\nabla^M_y\left(\frac{\nabla^M r}{\vert \nabla^M r \vert}\right)\rangle\frac{\nabla^M r}{\vert \nabla^M r \vert}=-\Hess^Mr(x,y)\frac{\nabla^Mr}{\vert \nabla^Mr\vert^2}\\
=&-\left(\frac{C_{\kappa}}{S_{\kappa}}(t)\langle x, y \rangle+\langle \nabla^{\perp}r,\alpha(x,y)\rangle\right)\frac{\nabla^M r}{\vert \nabla^M r \vert^2}\cdot
\end{aligned}
\end{equation}

Therefore,

  \begin{equation}\label{llarga}
  \begin{aligned}
  K_{\partial V(t)}(\pi)=&\kappa +\langle\alpha(e_i,e_i),\alpha(e_j,e_j)\rangle-\vert\alpha(e_i,e_j)\vert^2\\&+  \Biggl\{\left(\frac{C_{\kappa}}{S_{\kappa}}(t)+\langle \nabla^{\perp}r,\alpha(e_i,e_i)\rangle\right)\left(\frac{C_{\kappa}}{S_{\kappa}}(t)+\langle \nabla^{\perp}r,\alpha(e_j,e_j)\rangle\right)\Biggr.\\&\Biggl.- \langle \nabla^{\perp}r,\alpha(e_j,e_j)\rangle^2\Biggr\}\frac{1}{\vert \nabla^M r \vert^2}
  \end{aligned}
  \end{equation}

Since the immersion is tamed, we have,  for $t$ large enough 
\begin{equation}
\begin{aligned}
&\frac{C_{\kappa}}{S_{\kappa}}(t)+\langle \nabla^{\perp}r,\alpha(e_i,e_i)\rangle\geq \frac{C_{\kappa}}{S_{\kappa}}(t)-\vert\nabla^{\perp}r\vert\, \Vert\alpha\Vert>0,\\
&\\
&\frac{C_{\kappa}}{S_{\kappa}}(t)+\langle \nabla^{\perp}r,\alpha(e_j,e_j)\rangle\geq \frac{C_{\kappa}}{S_{\kappa}}(t)-\vert\nabla^{\perp}r\vert\, \Vert\alpha\Vert>0.
\end{aligned}
\end{equation}
Therefore the upper bounds on the statement of the proposition follows directly from the identity (\ref{llarga}). In order to obtain the lower bounds, observe that from equality (\ref{llarga})

\begin{eqnarray}
K_{\partial V(t)}(\pi)&\geq& \kappa -2\Vert \alpha \Vert^2+\frac{\left(\frac{C_\kappa}{S_\kappa}(t)-\vert \nabla^\perp r\vert\, \Vert\alpha\Vert\right)^2-\left(\nabla^\perp r\vert\,\Vert\alpha\Vert\right)^2}{\vert \nabla^M r\vert^2}\nonumber \\
&&\\
&\geq & \kappa -2\Vert \alpha \Vert^2+\frac{\left(\frac{C_\kappa}{S_\kappa}(t)\right)^2-2\vert \nabla^\perp r\vert\, \Vert\alpha\Vert\frac{C_\kappa}{S_\kappa}(t)}{\vert \nabla^M r\vert^2}.\nonumber
\end{eqnarray}

%
And the proposition follows.

\end{proof}

\section{Proof of theorem \ref{Euclidean1}}\label{section-proof-euc}

As we have observed in the Introduction, the assertion (1) in Theorem \ref{Euclidean1} follows from Theorem \ref{tamed-theorem}. 

In order to prove the assertion (2), let us remind that by theorem \ref{tamed-theorem}, since $\varphi\colon M\to \erre^n$ is a  tamed immersion, there exists $R_0 >0$ such that $M$ has finitely many ends $V_k\in M\setminus D_{R_0}$ and  we can work on each end separately. Let us denote
\begin{equation}
\partial V_k(t):=\partial D_t\cap V_k.
\end{equation}  

Applying too Theorem \ref{tamed-theorem}, we have that for any $t>R_0$ the extrinsic distance function has no critical points in 
$$
A^k_{R_0,t}:=\overline{\left(V_k\cap D_t\right)\setminus \left(V_k\cap D_{R_0}\right)},
$$  
so using basic Morse theory (see \cite[theorem 2.3]{Sakai} and \cite{Milnor} ), we know that $A^k_{R_0,t}$ is diffeomorphic to $\partial V_k(R_0)\times [R_0,t]$. In particular, $\partial V_k(t)$ is diffeomorphic to  $\partial V_k(R_0)$ for any $t>R_0$. Hence, by statement (4) of theorem \ref{tamed-theorem} for any $t\geq R_0$
\begin{equation}\label{diffeo}
 V_k\overset{\text{ diffeo. }}{\approx}\partial V_k(t)\times [0,\infty).
\end{equation}

Since $a(M)<\frac{1}{2}  <1$,  using Proposition \ref{curvature-comp}, there exists $t_0 >0$ such that the sectional curvatures of the tangent planes $\pi$ to $\partial V_k(t)$ (for all $t>t_0$) are bounded below and above by
%
 \begin{equation*}
 \frac{1}{t^2}\left(t^2\Vert\alpha\Vert^2+\frac{(1+\vert\nabla^\perp r\vert t\Vert \alpha \Vert)^2}{\vert \nabla^M r\vert^2}\right)\geq 
K_{\partial V_k(t)}\geq \frac{1}{t^2}\left(-2t^2\Vert\alpha\Vert^2+\frac{1-2\vert\nabla^\perp r\vert t \Vert \alpha \Vert}{\vert \nabla^M r\vert^2}\right)\nonumber
\end{equation*}
 Let us consider now a quantity $c \in (a(M), \frac{1}{2})$. From the definition of $a(M)$  there exist $t_c$ such
\begin{equation}
t\Vert \alpha \Vert \leq c,
\end{equation} 
for all $t>t_c$. Therefore, for any $t>\max\{t_c,t_0, R_0\}$
\begin{equation}\label{Gimeno-Palmer}
\frac{1}{t^2}\left(c^2+\frac{(1+\vert\nabla^\perp r\vert c)^2}{\vert \nabla^M r\vert^2}\right)\geq
K_{\partial V_k(t)}\geq \frac{1}{t^2}\left(-2c^2+\frac{1-2\vert\nabla^\perp r\vert c}{\vert \nabla^M r\vert^2}\right)
\end{equation} 
Taking into account that $t\Vert \alpha \Vert \leq c$, the inequalities \ref{Gimeno-Palmer} yields
\begin{equation}
\frac{1}{t^2}\left(c^2+\frac{(1+\vert\nabla^\perp r\vert c)^2}{1-\vert \nabla^\perp r\vert^2}\right)\geq 
K_{\partial V_k(t)}\geq \frac{1}{t^2}\Bigl(1-2c^2-2\vert\nabla^\perp r\vert c\Bigr),
\end{equation}
for any $t>\max\{t_c,t_0,R_0\}$.
Applying 
Lemma \ref{kasue} for $G(t)=\displaystyle c/t$ we have, for any $t>\max\{t_c,t_0,R_0\}$, that 
\begin{equation}\label{grad-ine}
\vert \nabla^\perp r \vert \leq \delta(t)+c(1-\frac{R_0}{t})\leq \delta(t)+c.
\end{equation}
and  
\begin{equation}\label{quasi-final}
\frac{1}{t^2}\left(c^2+\frac{\left(1+c(\delta(t)+c)\right)^2}{1-\left(\delta(t)+c)\right)^2}\right)\geq 
K_{\partial V_k(t)}\geq \frac{1}{t^2}\left(1-2c^2-2c\left(\delta(t)+c\right)\right)
\end{equation}
%

We are dealing with lower bounds in order to prove statement (2) of Theorem \ref{Euclidean1}. Since  $\delta(t)\to 0$, when $t\to\infty$, and $c < \frac{1}{2}$, there exist $t_1>\max\{t_c,t_0, R_0\}$ such that 
\begin{equation}
1-2c^2-2c\left(\delta(t)+c\right)>0
\end{equation}
 for any $t>t_1$.

Defining the function $\Lambda^0\colon\erre_+\to \erre$ as 
\begin{equation}
\Lambda_c^0(t)\colon=1-2c^2-2c\left(\delta(t)+c\right),
\end{equation}
the lower bound for  $K_{\partial V_k(t)}$ in  inequality (\ref{quasi-final}) can be therefore written as
\begin{equation}\label{curv-baix}
\begin{aligned}
K_{\partial V_k(t)}\geq \frac{\Lambda_c^0(t)}{t^2}\geq\frac{\Lambda_c^0(t_1)}{t^2} >\,0,
\end{aligned}
\end{equation} 
for any $t>t_1$.

Now, we apply Bishop's volume comparison theorem (see \cite{Chavel} or \cite{Chavel2}), taking into account that the above inequality implies for any  unit vector $\xi \in T_p\partial V_k(t)$,
\begin{equation}
 \begin{aligned}
  \text{Ricc}_{\partial V_k(t)}(\xi,\xi)\geq &(m-2)\frac{\Lambda^0_c(t_1)}{t^2},
  \end{aligned}
\end{equation}
we conclude that
\begin{equation}
\vol (\partial V_k(t))\leq \left(\frac{1}{\Lambda^0_c(t_1)}\right)^{\frac{m-1}{2}}m\omega_mt^{m-1}.
\end{equation}

Since
\begin{equation}
\partial D_t=\overset{\E(M)}{\underset{k=1}{\bigcup}}\partial V_k(t),
\end{equation}
one concludes
\begin{equation}\label{sphere-comp}
\vol (\partial D_t)\leq \left(\frac{1}{\Lambda^0_c(t_1)}\right)^{\frac{m-1}{2}}\,\E(M)\, m\,\omega_m\,t^{m-1}.
\end{equation}
for any $t>t_1$.
Applying coarea formula (see for instance \cite{Sakai}) to the extrinsic annuls $A_{t_1,t}:=D_t\setminus D_{t_1}$, and using inequality (\ref{grad-ine})

\begin{equation}
\begin{aligned}
\vol(A_{t_1,t})&=\int_{t_1}^t\int_{\partial D_s}\left(\frac{1}{\vert \nabla^M r \vert }dV\right)ds\\
&=\int_{t_1}^t\int_{\partial D_s}\left(\frac{1}{\sqrt{1-\vert \nabla^\perp r \vert^2} }dV\right)ds\\
&\leq  \left(\frac{1}{1-(c+\delta(t_1))^2}\right)^{\frac{1}{2}} \int_{t_1}^t\vol(\partial D_s)ds\\
&\leq   \frac{\E(M)\omega_m\left(t^{m}-t_1^m\right)}{\left[1-(c+\delta(t_1))^2\right]^{\frac{1}{2}}\left(\Lambda^0_c(t_1)\right)^{\frac{m-1}{2}}}
\end{aligned}
\end{equation}

therefore,

\begin{equation}\label{ball-comp}
\vol(D_t)\leq \vol(D_{t_1})+\frac{\E(M)\omega_m\left(t^{m}-t_1^m\right)}{\left[1-(c+\delta(t_1))^2\right]^{\frac{1}{2}}\left(\Lambda^0_c(t_1)\right)^\frac{m-1}{2}}
\end{equation}

Taking limits in inequalities (\ref{sphere-comp}) and (\ref{ball-comp})
\begin{equation}
\begin{array}{rll}
\liminf_{t\to\infty}\displaystyle\frac{\vol(\partial D_t)}{m \omega_m t^{m-1}}&\leq& \displaystyle\left(\frac{1}{\Lambda^0_c(t_1)}\right)^{\frac{m-1}{2}}\E(M)
\\
&&\\
\liminf_{t\to\infty}\displaystyle\frac{\vol(D_t)}{\omega_m t^m}& \leq & \displaystyle\frac{\E(M)}{\left[1-(c+\delta(t_1))^2\right]^{\frac{1}{2}}\left(\Lambda^0_c(t_1)\right)^\frac{m-1}{2}}
\end{array}
\end{equation}

Letting $t_1\to \infty$ and taking into account $\Lambda^0_c(t_1)\to 1-4c^2$ we have that 
\begin{equation}
\begin{array}{rll}
\liminf_{t\to\infty}\displaystyle\frac{\vol(\partial D_t)}{m \omega_m t^{m-1}}&\leq& \left(\displaystyle\frac{1}{1-c^2}\right)^{\frac{m-1}{2}}\E(M)\\
&&\\
\liminf_{t\to\infty}\displaystyle\frac{\vol(D_t)}{\omega_m t^m} &\leq&\displaystyle \frac{\E(M)}{\left[1-c^2\right]^{\frac{1}{2}}\left(1-4c^2\right)^\frac{m-1}{2}}
\end{array}
\end{equation}
 Since the above inequalities are true for any $c\in (a(M), \frac{1}{2})$ the desired inequalities of statement (2) of the Theorem \ref{Euclidean1} follow when $c$ goes to $a(M)$.
\bigskip

In order to prove statements (3) and (4) of the theorem, let us define, for $t>0$ and for all $c \in [a(M), \frac{1}{2})$
$$
K_{max}(t):=\frac{1}{t^2}\left(c^2+\frac{\left(1+c(\delta(t)+c)\right)^2}{1-\left(\delta(t)+c)\right)^2}\right).
$$
and 
$$
K_{min}(t):=\frac{1}{t^2}\left(1-2c^2-2c\left(\delta(t)+c\right)\right).
$$

We shall prove that, for $t$ large enough and when $a(M) <\left[\frac{23-\sqrt{337}}{32}\right]^\frac{1}{2} < \frac{1}{2}$, the sectional curvatures of the boundary of each end, $\partial V_k(t):=\partial D_t\cap V_k$ satisfy the pinching:

\begin{equation}\label{sphere-type}
K_{max}(t)\geq K_{\partial V_k(t)}> \frac{1}{4}K_{max}(t)>0.
\end{equation}

\noindent and then, we apply Synge's Theorem and either the Rauch-Berger Sphere Theorem or the Brendle-Schoen differentiable sphere theorem, if $m\geq 5$, splitting the proof in two cases, according to parity dimension of $\partial V_k(t)$. First of all, we know that, in all cases, $\partial V_k(t)$ is orientable, because there exist a everywhere non vanishing smooth normal vector field $\frac{\nabla^M r}{\vert \nabla^M r\vert}$ globally defined on $\partial V_k(t)$.

In assertion (3), we assume that dimension $m$ of the submanifold $M$ is odd, so $\partial V_k(t)$ is even dimensional. By Synge's Theorem (see \cite[Corollary 3.10, Chap. 9]{DoCarmo2}), $\partial V_k(t)$ is simply connected. Taking into account the inequality (\ref{sphere-type}) and the Rauch-Berger Sphere Theorem (see \cite[Theorem 1.1, Chap. 13]{DoCarmo2}), $\partial V_k(t)$ is homeomorphic to $\mathbb{S}^{m-1}$. 
If $m-1\geq 4$ we apply  Brendle-Schoen Differentiable Sphere Theorem, see \cite{Brendle-Schoen}, to see that $\partial V_k(t)$ is diffeomorphic to $\mathbb{S}^{m-1}$.   Since   $V_k\overset{\text{ diffeo. }}{\approx}\partial V_k(t)\times [0,\infty)$, see equation \eqref{diffeo}, the statement (3) of Theorem \ref{Euclidean1} is proven.

In assertion (4),  we assume that dimension $m$ of the submanifold $M$ is even, so $\partial V_k(t)$ is odd dimensional. Moreover, we assume that each end $V_k\overset{\text{ diffeo. }}{\approx}\partial V_k(t)\times [0,\infty)$ is simply connected, so also it is $\partial V_k(t)$. As $\partial V_k(t)$ is also orientable, we apply either the Rauch-Berger Sphere Theorem or  Brendle-Schoen Differentiable Sphere Theorem, observing that in order to have $m-1\geq 4 $ and $m$ even  then $m\geq 6$, to obtain the proof of assertion (4). 

We are going now to prove that the sectional curvatures of $\partial V_k(t)$ are pinched as in (\ref{sphere-type}).
 First of all, observe that,  given any $c\in [a(M), \frac{1}{2})$, as $a(M) \leq c$, there exist $t_c$ such
\begin{equation}
t\Vert \alpha \Vert \leq c,
\end{equation} 
for all $t>t_c$. Therefore, for any $t>\max\{t_c,t_0, R_0\}$, we can repeat the argument that leads to inequalities (\ref{quasi-final}) to obtain, for all $c\in [a(M), \frac{1}{2})$:
\begin{equation}\label{quasi-final2}
\frac{1}{t^2}\left(c^2+\frac{\left(1+c(\delta(t)+c)\right)^2}{1-\left(\delta(t)+c)\right)^2}\right)\geq 
K_{\partial V_k(t)}\geq \frac{1}{t^2}\left(1-2c^2-2c\left(\delta(t)+c\right)\right)
\end{equation}
%
%
for any $t>\max\{t_c,t_0, R_0\}$.
Define, for any $c\in [a(M),\frac{1}{2})$ and any $t>\max\{t_c,t_0, R_0\}$ the function
\begin{equation*}
F(c,t):= \displaystyle\frac{1-2c^2-2c\left(\delta(t)+c\right)}{c^2+\displaystyle\frac{\left(1+c(\delta(t)+c)\right)^2}{1-\left(\delta(t)+c)\right)^2}}
\end{equation*}

\noindent Hence, by inequality (\ref{quasi-final2})
\begin{equation}
K_{max}(t)\geq K_{\partial V_k(t)}\geq K_{min}(t)=F(c,t)K_{max}(t)
\end{equation}
for all $t>\max\{t_c,t_0, R_0\}$.
On the other hand, we have that, for all $c\in [a(M),\frac{1}{2})$,
\begin{equation*}
F(c,\infty)= \lim_{t\to\infty}\displaystyle\frac{1-2c^2-2c\left(\delta(t)+c\right)}{c^2+\displaystyle\frac{\left(1+c(\delta(t)+c)\right)^2}{1-\left(\delta(t)+c)\right)^2}}=\frac{1-4c^2}{c^2+\displaystyle\frac{\left(1+c^2\right)^2}{1-c^2}}\cdot
\end{equation*}

It is straightforward that $0 < F(c,\infty) \leq 1$ for all $c \in [0, \frac{1}{2})$. On the other hand, $\frac{d}{dc} F(c, \infty) <0, \,\,\,\forall c < \frac{1}{2}$, as it is easy to check, the function $F(c,\infty)$ is strictly decreasing  in $c \in [0, \frac{1}{2})$. Hence, let us choose $c^{*}=\left[\frac{23-\sqrt{337}}{32}\right]^\frac{1}{2}< \frac{1}{2}$ such that $F(c^*, \infty)=\frac{1}{4}\cdot$
\smallskip

Then, for all $c \in [a(M), c^*)$, $F(c,\infty) > F(c^*, \infty)=\frac{1}{4}$.
Let us fix  $c_0 \in [a(M), c^*)$. Then, given $\epsilon >0$, there exists $t_{\epsilon}$ such that for all $t > t_{\epsilon}$,
\begin{equation}
F(c_0,t) \geq F(c_0,\infty)-\epsilon
\end{equation}
thus, for any $t>\max\{t_{\epsilon},t_{c_0},t_0, R_0\}$, we have
\begin{equation}\label{ineqs}
K_{max}(t)\geq K_{\partial V_k(t)}\geq K_{min}(t)=F(c_0,t)K_{max}(t) \geq (F(c_0,\infty)-\epsilon)K_{max}(t).
\end{equation}

Let us choose 
$$\epsilon:= F(c,\infty)-F(c^*,\infty) >0$$ and, for this $\epsilon$, for $t$ large enough, we have, from (\ref{ineqs}),

\begin{equation}\label{ineqs2}
K_{max}(t)\geq K_{\partial V_k(t)} > \frac{1}{4}K_{max}(t)
\end{equation} and the sectional curvature pinching of $\partial V_k(t)$ is proven.

Moreover, if $a(M)<\frac{1}{2}$ and $m=3$ by using inequality (\ref{curv-baix}) and the Gauss-Bonnet Theorem, since the surface $\partial V_k(t)$ has positive curvature, the surface $\partial V_k(t)$ is homeomorphic to a sphere, and we can state

\begin{corollary}\label{dim3-theo}
Let $\varphi\colon M^3
\hookrightarrow \erre^n$ be an isometric immersion of a complete Riemannian
$3$-manifold $M$ into a $n$-dimensional Euclidean space $\erre^n$. Then, if  $a(M)<\frac{1}{2}$, each end of $M$ is homeomorphic to $\mathbb{S}^{2}\times \erre$.
\end{corollary}
Actually, if $a(M)<\frac{1}{2}$ inequality (\ref{curv-baix}) implies for any dimension $m>2$ that the sectional curvature of $\partial V_k(t)$ is positive. By the first Betti number Theorem (see \cite{Jost}) and corollary 2.5 of \cite{Gromov-Spin} we can state
\begin{corollary}\label{betti}
Let $\varphi\colon M^m
\hookrightarrow \erre^n$ be an isometric immersion of a complete Riemannian
$m$-manifold $M$ into a $n$-dimensional Euclidean space $\erre^n$. Then, if  $a(M)<\frac{1}{2}$, each end $V_k(M)$ of $M$ is homeomorphic to $\partial V_k\times \erre$, where $\partial V_k$ is a $m-1$ compact manifold such that
\begin{enumerate}
\item $\partial V_k$ has zero first Betti number, $b_1(\partial V_k)=0$.
\item $\partial V_k$ is not homeomorphic to the $m-1$ torus $T^{m-1}$.
\end{enumerate} 
\end{corollary}

\section{Proof of theorem \ref{Hyperbolic1}}\label{section-proof-hyp}

This proof follows the lines of the proof in Section \S \ref{section-proof-euc}. 
As we have observed in \S \ref{section-proof-euc}, assertion (1) of the Theorem follows from Theorem \ref{tamed-theorem}. 

In order to prove assertion (2), we have, as before, a finite number of ends $V_k\in M\setminus D_{R_0}$ with boundaries $\partial V_k(t):=\partial D_t\cap V_k$ and  we work on each end separately. 
Taking into account that for any $t>R_0$ the extrinsic distance function has no critical points in 
$$
A^k_{R_0,t}:=\overline{\left(V_k\cap D_t\right)\setminus \left(V_k\cap D_{R_0}\right)},
$$  
we use Morse theory (see \cite[theorem 2.3]{Sakai} and \cite{Milnor} ), to have that $A^k_{R_0,t}$ is diffeomorphic to $\partial V_k(R_0)\times [R_0,t]$. In particular, $\partial V_k(t)$ is diffeomorphic to  $\partial V_k(R_0)$ for any $t>R_0$. Hence, by statement (4) of theorem \ref{tamed-theorem} for any $t\geq R_0$
\begin{equation}\label{diffeo2}
 V_k\overset{\text{ diffeo. }}{\approx}\partial V_k(t)\times [0,\infty).
\end{equation}

Since $b(M) < \infty$, then $a(M) <1$, so using again Proposition \ref{curvature-comp}, there exists $t_0 >0$ such that

\begin{equation}
K_{\partial V_k(t)}(\pi)\geq \frac{1}{S_\kappa^2}(t)-\left(\frac{C_\kappa}{S_\kappa}(t)\right)^2 -2\Vert \alpha \Vert^2+\frac{\left(\frac{C_\kappa}{S_\kappa}(t)\right)^2-2\vert \nabla^\perp r\vert\, \Vert\alpha\Vert\frac{C_\kappa}{S_\kappa}(t)}{\vert \nabla^M r\vert^2}.
\end{equation}
for all $t>\max\{t_0, R_0\}$. Hence, 
\begin{equation}
\begin{aligned}
K_{\partial V_k(t)}(\pi)\geq   \left(\frac{1}{S_\kappa(t)}\right)^2\Biggl(\Biggr. & \left. 1-\left({C_\kappa}(t)\right)^2 -  2\left({S_\kappa}(t)\Vert \alpha \Vert\right)^2\right.\\
& +\left.\frac{\left({C_\kappa}(t)\right)^2-2\vert \nabla^\perp r\vert\, \Vert\alpha\Vert{C_\kappa}{S_\kappa}(t)}{\vert \nabla^M r\vert^2}\right).
\end{aligned}
\end{equation}
 
Let $b^*$ be such that $b(M)<b^*<\infty$. For any $c\in (b(M),b^*)$ there exist therefore $t_c$ such that, for all $t >t_c$,
\begin{equation}
\Vert \alpha \Vert \leq \frac{c}{S_\kappa(t)C_\kappa(t)}
\end{equation}

\noindent and hence, for any $t> \max\{t_c,t_0, R_0\}$,

\begin{equation}
\begin{aligned}
K_{\partial V_k(t)}(\pi)\geq \left(\frac{1}{S_\kappa(t)}\right)^2\left(1-\left({C_\kappa}(t)\right)^2 -2\left(\frac{c}{{C_\kappa}(t)}\right)^2\right.
\left.+\frac{\left({C_\kappa}(t)\right)^2-2\vert \nabla^\perp r\vert\, c}{\vert \nabla^M r\vert^2}\right)\\=\left(\frac{1}{S_\kappa(t)}\right)^2\left(1 -2\left(\frac{c}{{C_\kappa}(t)}\right)^2\right.
\left.+\frac{\left({C_\kappa}(t)\right)^2\vert \nabla^\perp r\vert^2-2\vert \nabla^\perp r\vert\, c}{\vert \nabla^M r\vert^2}\right).
\end{aligned}
\end{equation}


As $\vert \nabla^M r\vert^2 \leq 1$ and  $\left({C_\kappa}(t)\right)^2\vert \nabla^\perp r\vert^2 -2c \vert \nabla^\perp r\vert\geq 0$, we have, for all $t>\max\{t_c,t_0, R_0\}$,

\begin{equation}
\frac{\left({C_\kappa}(t)\right)^2\vert \nabla^\perp r\vert^2-2\vert \nabla^\perp r\vert\, c}{\vert \nabla^M r\vert^2}\geq {C_\kappa}(t)^2\vert \nabla^\perp r\vert^2-2\vert \nabla^\perp r\vert\, c
\end{equation}

and hence, as ${C_\kappa}(t)^2\vert \nabla^\perp r\vert^2 \geq 0$ too,

\begin{equation}
K_{\partial V_k(t)}(\pi)\geq \left(\frac{1}{S_\kappa(t)}\right)^2\left(1-2\left(\frac{c}{{C_\kappa}(t)}\right)^2-2\vert \nabla^\perp r\vert\, c\right).
\end{equation}

On the other hand, as $b(M) < b^* < \infty$ and $c\in (b(M),b^*)$, we also have that for all $t >\max\{t_c,t_0, R_0\}$, 
\begin{equation}
\Vert \alpha \Vert \leq \frac{c}{S_\kappa(t)C_\kappa(t)}\leq \frac{c}{S_\kappa(t)}
\end{equation}

\noindent so applying Lemma \ref{kasue} with $G(t)=c/S_\kappa(t)$ for $t> \max\{t_c,t_0, R_0\}$ we obtain\begin{equation}
\vert \nabla^\perp r\vert \leq \delta(t)+c\, \frac{t-R_0}{S_\kappa(t)}=u_c(t).
\end{equation}

Since $\delta(t)$ is a decreasing function and $\gamma(t)\colon=\displaystyle c\cdot \frac{t-R_0}{S_\kappa(t)}$ is also a decreasing function for a sufficiently large $t$, the right side function $u_c(t):=\delta(t)+c\frac{t-R_0}{S_\kappa(t)}$ is  decreasing for  a sufficiently large $t$, so there exists $t_1 > \max\{t_c,t_0, R_0\}$ such that, for all $t >t_1$ we have

\begin{equation}\label{gradient2}
\vert \nabla^\perp r\vert \leq u_c(t) \leq u_c(t_1).
\end{equation}
and hence one obtains,
\begin{equation}\label{curvhyp}
\begin{aligned}
K_{\partial V_k(t)}(\pi)\geq &\left(\frac{1}{S_\kappa(t)}\right)^2\Lambda_c(t_1)\, >0.
\end{aligned}
\end{equation}
where now
\begin{equation}
\Lambda_c(t)=1-2\left(\frac{c}{{C_\kappa}(t)}\right)^2-2c\,u_c(t).
\end{equation}

By using the Bishop Volume Comparison Theorem, one concludes that, for all $t> t_1$
\begin{equation}
\vol (\partial V_k(t))\leq \left(\frac{1}{\Lambda_c(t_1)}\right)^{\frac{m-1}{2}}\vol(S_t^{\kappa,m-1}),
\end{equation}
obtaining therefore the following lower estimate for the number of ends
\begin{equation}\label{sphere-comp-hyp}
\vol (\partial D_t)\leq \left(\frac{1}{\Lambda_c(t_1)}\right)^{\frac{m-1}{2}}\E(M)\vol(S_t^{\kappa,m-1}).
\end{equation}

Similarly that in \S \ref{section-proof-euc} by using again the coarea formula

\begin{equation}\label{ball-comp-2}
\begin{aligned}
\vol(D_t)\leq \vol(D_{t_1})+\frac{\E(M)\vol(B_t^{\kappa,m}\setminus B_{t_1}^{\kappa,m})}{\left[1-(u(t_1))^2\right]^{\frac{1}{2}}\left(\Lambda_c(t_1)\right)^\frac{m-1}{2}}
\end{aligned}
\end{equation}

And therefore, taking limits in inequalities (\ref{sphere-comp-hyp}) and (\ref{ball-comp-2})

\begin{equation}
\begin{aligned}
\liminf_{t\to\infty}&\frac{\vol(\partial D_t)}{\vol(S_t^{\kappa,m-1})}\leq \left(\frac{1}{\Lambda_c(t_1)}\right)^{\frac{m-1}{2}}\E(M)\\
\liminf_{t\to\infty}&\frac{\vol(D_t)}{\vol(B_t^{\kappa,m})} \leq \frac{\E(M)}{\left[1-(u(t_1))^2\right]^{\frac{1}{2}}\left(\Lambda_c(t_1)\right)^\frac{m-1}{2}}.
\end{aligned}
\end{equation}

Letting $t_1\to \infty$ in the above inequalities and taking into account that
\begin{equation}
\begin{aligned}
&\lim_{t\to\infty}u(t)=0\\
&\lim_{t\to\infty}\Lambda_c(t)=1.
\end{aligned}
\end{equation}
the desired inequalities of statement (2) of theorem \ref{Hyperbolic1} follow.

\smallskip
In order to obtain the upper bound for the fundamental tone we only have to take into account that the geodesic ball $B_t^M$ of radius $t$ is a subset of the extrinsic ball $D_t$ of the same radius $t$ because the extrinsic distance is always less or equal to the geodesic intrinsic distance in isometric immersions. Hence, by inequality (\ref{ball-comp-2})

\begin{equation}
\begin{aligned}
\vol(B_t^M)\leq \vol(D_{t_1})+\frac{\E(M)\vol(B_t^{\kappa,m}\setminus B_{t_1}^{\kappa,m})}{\left[1-(u(t_1))^2\right]^{\frac{1}{2}}\left(\Lambda_c(t_1)\right)^\frac{m-1}{2}}
\end{aligned}
\end{equation}

then

\begin{equation}
\limsup_{t\to\infty}\frac{\log\left(\vol(B^M_t(x_0)\right)}{t}\leq \left(m-1\right)\sqrt{-\kappa}
\end{equation}
Finally by using \cite[inequality (10.3)]{GriExp} the upper bound is proved.

Moreover and in the same way than in Theorem \ref{Euclidean1}, if $b(M)<\infty$ and $m=3$ by using inequality (\ref{curvhyp}) and the Gauss-Bonnet Theorem, since the surface $\partial V_k(t)$ has positive curvature for a sufficiently large $t$, the surface $\partial V_k(t)$ is homeomorphic to an sphere, and we can state

\begin{corollary}\label{dim3-theoHyp}
Let $\varphi\colon\! M^3\!
\hookrightarrow \Han$ be an isometric immersion of a complete Riemannian
$3$-manifold $M$ into a $n$-dimensional Hyperbolic space $\Han$. Then, if  $b(M)<\infty$, each end of $M$ is homeomorphic to $\mathbb{S}^{2}\times \erre$.
\end{corollary}
Actually, if $b(M)<\infty$ then inequality (\ref{curvhyp}) implies for any dimension $m>2$ that the sectional curvature of $\partial V_k(t)$ is positive for a sufficiently large $t$. As in Corollary \ref{betti}, by the first Betti number Theorem (see \cite{Jost}) and corollary 2.5 of \cite{Gromov-Spin} we can state
\begin{corollary}\label{bettiHyp}
Let $\varphi\colon\! M^m\!
\hookrightarrow \!\Han$ be an isometric immersion of a complete Riemannian
$m$-manifold $M$ into a $n$-dimensional Hyperbolic space $\Han$. Then, if  $b(M)<\infty$, each end $V_k(M)$ of $M$ is homeomorphic to $\partial V_k\times \erre$, where $\partial V_k$ is a $m-1$ compact manifold such that
\begin{enumerate}
\item $\partial V_k$ has zero first Betti number, $b_1(\partial V_k)=0$.
\item $\partial V_k$ is not homeomorphic to the $m-1$ torus $T^{m-1}$.
\end{enumerate} 
\end{corollary}

\section{Proof of Theorem \ref{gap1}}\label{section-proof-gap}

The submanifold  $M^m$ is simply connected and it has sectional curvatures bounded from above by $K_M\leq k \leq  0$. Then, all the points $p \in M$ are poles of $M$ and the number of ends of $M$ is $\E(M)=1$.  We apply Bishop-G\"unther's Theorem (see for instance \cite{Chavel2}), so we have, for any geodesic ball $B_R^M(p)$ of radius $R$ on $M$,
\begin{equation}\label{Bishop-Gunter}
1\leq \frac{\vol(B_R^M)}{\vol(B^{\kam}_{R})}.
\end{equation}
where, we recall that,
$B^{\kam}_{R}$ denotes the open geodesic ball
of radius $R$ in the real space form of constant sectional curvature $k \leq 0$.
Moreover,  
\[R\longrightarrow\frac{\vol(B_R^M)}{\vol(B^{\kam}_{R})}\] is an non-decreasing function on $R$ and equality in \eqref{Bishop-Gunter}  is attained  if and only if $B_R^M$ is isometric to the geodesic ball of the same radius $R$ in $\kam$. Taking into account that $B_R^M\subset D_R$, we have, for any $t >R$
\begin{equation}
\begin{aligned}
1\leq \frac{\vol(B_R^M)}{\vol(B^{\kam}_{R})}\leq\frac{\vol(B_t^M)}{\vol(B^{\kam}_{t})}\leq\frac{\vol(D_t)}{\vol(B^{\kam}_{t})}
\end{aligned}
\end{equation}

From now on, we split the proof in two cases.

\smallskip

When $k=0$, we use inequality (\ref{ball-comp}) and the fact that $\E(M)=1$ to get:

\begin{equation}
\begin{aligned}
1\leq \frac{\vol(B_R^M)}{\omega_m R^m}\leq&\frac{\vol(B_t^M)}{\omega_m t^m}\leq\frac{\vol(D_t)}{\omega_m t^m}\\
\leq &\frac{\vol(D_{t_1})}{\omega_mt^m}+\frac{\left(1-(\frac{t_1}{t})^m\right)}{\left[1-(c+\delta(t_1))^2\right]^{\frac{1}{2}}\left(\Lambda^0_c(t_1)\right)^\frac{m-1}{2}}
\end{aligned}
\end{equation}

Letting $t\to \infty$, and then $t_1\to\infty$,
\begin{equation}
\begin{aligned}
1\leq \frac{\vol(B_R^M)}{\omega_m R^m}\leq\frac{1}{\left[1-c^2\right]^{\frac{1}{2}}\left(1-4c^2\right)^\frac{m-1}{2}}
\end{aligned}
\end{equation}

Finally taking $c\to 0$,

\begin{equation}
\begin{aligned}
\frac{\vol(B_R^M)}{\omega_m R^m}=1,
\end{aligned}
\end{equation}
for any $R\in \erre_+$ and that completes the proof of the corollary, because then, for all the points $p\in M$, any geodesic $R$- ball $B^M_R(p)$ is isometric $B^{\erre^n}_R$.

When $k < 0$, we argue exactly in the same way, but using now inequality (\ref{ball-comp-2}).

\def\cprime{$'$} \def\cprime{$'$} \def\cprime{$'$} \def\cprime{$'$}
  \def\cprime{$'$}
  

\end{document}